\documentclass[11pt]{article}

\usepackage[T1]{fontenc}
\usepackage{amsmath}
\usepackage{amssymb}
\usepackage{amscd}
\usepackage{amsbsy}
\usepackage{amsfonts}
\usepackage{amsthm}
\usepackage{bbm}
\usepackage{mathtools}
\usepackage{xcolor}
\usepackage{comment}

\input xy
\xyoption{all}

\newtheorem{thm}{Theorem}[section]
\newtheorem{lem}[thm]{Lemma}
\newtheorem{prop}[thm]{Proposition}
\newtheorem{cor}[thm]{Corollary}
\newtheorem{rem}[thm]{Remark}
\newtheorem{dfn}[thm]{Definition}
\newtheorem{conj}[thm]{Conjecture}

\DeclareMathOperator{\SH}{\mathcal{SH}}
\DeclareMathOperator{\cat}{\mathcal{C}}
\DeclareMathOperator{\giso}{{\mathcal{G}}^{iso}}
\DeclareMathOperator{\uiso}{{\mathcal{U}}^{iso}}
\DeclareMathOperator{\mbp}{\mathrm{MBP}}
\DeclareMathOperator{\hz}{\mathrm{H}{\mathbb Z}/2}
\DeclareMathOperator{\F}{{\mathbb F}_2}
\DeclareMathOperator{\Z}{{\mathbb Z}}
\DeclareMathOperator{\N}{{\mathbb N}}
\DeclareMathOperator{\Q}{{\mathbb Q}}
\DeclareMathOperator{\un}{{\mathbbm 1}}
\DeclareMathOperator{\Fi}{{\mathbb F}}
\DeclareMathOperator{\Pone}{{\mathbb P}^1}
\DeclareMathOperator{\gm}{{\mathbb G}_m}
\DeclareMathOperator{\gmn}{{\mathbb G}_m^{\times n}}

\DeclareMathOperator{\km}{\mathrm{k}}
\DeclareMathOperator{\Mod}{\mathrm{Mod}}

\DeclareMathOperator{\vs}{\mathrm{vec}}
\DeclareMathOperator{\Vs}{\mathrm{Vec}}
\DeclareMathOperator{\gr}{\mathrm{gr}}
\DeclareMathOperator{\spec}{\mathrm{Spec}}
\DeclareMathOperator{\ext}{\mathrm{Ext}}
\DeclareMathOperator{\cof}{\mathrm{cofib}}
\DeclareMathOperator{\ce}{\mathrm{\check{C}}}
\DeclareMathOperator{\colim}{\mathrm{colim}}

\DeclareMathOperator{\perf}{\mathrm{Perf}}

\title{\textsc{Isotropic motivic fundamental groups}}
\author{Fabio Tanania}
\date{}

\begin{document}

\maketitle

 \begin{abstract}\let\thefootnote\relax\footnote{The author acknowledges support (through Timo Richarz) by the European Research Council (ERC) under Horizon Europe (grant agreement n$^\circ$ 101040935), by the Deutsche Forschungsgemeinschaft (DFG, German Research Foundation) TRR 326 \textit{Geometry and Arithmetic of Uniformized Structures}, project number 444845124 and the LOEWE professorship in Algebra, project number LOEWE/4b//519/05/01.002(0004)/87.} 
 The main goal of this paper is to study relative versions of the category of modules over the isotropic motivic Brown-Peterson spectrum, with a particular emphasis on their cellular subcategories. Using techniques developed by Levine, we equip these categories with motivic $t$-structures, whose hearts are Tannakian categories over $\Fi_2$. This allows to define isotropic motivic fundamental groups, and to interpret relative isotropic Tate motives in the heart as their representations. Moreover, we compute these groups in the cases of the punctured projective line and split tori. Finally, we also apply Spitzweck's derived approach to establish an identification between relative isotropic Tate motives and representations of certain affine derived group schemes, whose 0-truncations coincide with the aforementioned isotropic motivic fundamental groups.
 	\end{abstract}
 
 
\section{Introduction}
 
The triangulated category of isotropic motives with $\Fi_p$-coefficients $\mathcal{DM}(k/k,\Fi_p)$ was introduced by Vishik in \cite{V}. In a few words, it is a localization of Voevodsky's triangulated category of motives $\mathcal{DM}(k,\Fi_p)$ \cite{Vo}, obtained by ``annihilating" the motives of all anisotropic varieties mod $p$, that is, varieties whose closed points have degree divisible by $p$. At least when the base field $k$ is flexible, namely a purely transcendental extension of infinite degree over some other field, these categories exhibit an exceptionally nice behaviour. One of their most interesting features is their relation to numerical motives. More precisely, the category of pure isotropic motives is equivalent to the category of numerical motives mod $p$. This was first conjectured in \cite{V}, and then fully proved in \cite{V1}, by Vishik. The upshot is that isotropic motivic cohomology catches information about algebraic cycles modulo numerical equivalence mod $p$, as much as motivic cohomology catches information about Chow groups (modulo rational equivalence). 

Another remarkable characteristic of isotropic motives is their connection to well-known objects coming from topology, and classical homotopy theory. This relation is at the moment only visible for $p=2$, but is expected in general for every prime. The isotropic motivic cohomology of $k$ with $\Fi_2$-coefficients was computed in \cite{V}, and the outcome is an exterior algebra over $\Fi_2$ generated by classes which are, in some sense, dual to Milnor operations. This result provided motivation, in \cite{T1}, to consider a stable homotopic version of $\mathcal{DM}(k/k,\Fi_2)$, and to compute the stable motivic homotopy groups of the isotropic sphere spectrum $\un_k^{iso}$. Via a suitable isotropic Adams spectral sequence, the endomorphism ring of $\un_k^{iso}$ was identified with the Ext-algebra of the classical Steenrod algebra, that is, the $E_2$-page of the classical Adams spectral sequence.

As a follow-up, in \cite{T}, we described the cellular subcategory of the isotropic stable motivic homotopy category $\SH^{iso}(k)$. This was shown to be equivalent to the derived category of comodules over the classical dual Steenrod algebra. A key step in proving this identification was the investigation of the category of cellular modules over the isotropic motivic Brown-Peterson spectrum $\mbp^{iso}_k$, which is simply, quite surprisingly, the category of bigraded $\Fi_2$-vector spaces. Moreover, in \cite{T} the relation between isotropic motivic cohomology and isotropic $\mbp$-cohomology was made explicit. More precisely, the isotropic motivic cohomology of a smooth variety $X$ over $k$ is a free module over the isotropic $\mbp$-cohomology of $X$ generated by the classes appearing in the isotropic motivic cohomology of $k$. This implies, in particular, that algebraic cycles modulo numerical equivalence mod 2 are already visible in the category of $\mbp^{iso}_k$-modules. The part that is erased from $\mathcal{DM}(k/k,\Fi_2)$ is encoded in the isotropic motivic cohomology of the point, which is in some sense ``rigid". This translates into better vanishing properties of the endomorphisms of the unit in $\mbp_{k}^{iso}-\Mod$ than in $\mathcal{DM}(k/k,\Fi_2)$.

In this paper we are ultimately guided by the purpose of understanding isotropic $\mbp$-cohomology, since this theory generalizes cycles modulo numerical equivalence, and also provides the first step in understanding $\SH^{iso}(k)$. Our approach is centered on the study of relative isotropic Tate motives. More precisely, we introduce for every smooth variety $X$ over $k$ the category of $\mbp^{iso}_X$-modules, and for every finitely generated field extension $F/k$ the category of $\mbp_{F/k}$-modules. Then, we focus on their cellular subcategories, which already contain the information we are interested in. In particular, the endomorphism rings of the unit in these categories are respectively the isotropic $\mbp$-cohomology of $X$ and of $F$.\\

\textbf{Main results.}  Let $k=k_0(t_1,t_2,\dots)$ be a flexible field of characteristic different from 2. We denote by $\mbp_{k}^{iso}$ the isotropic localization of the motivic Brown-Peterson spectrum $\mbp$ at the prime 2. Recall from \cite{T} that $\mbp_{k}^{iso}$ is an $E_{\infty}$-ring spectrum in $\SH(k)$. Let $X$ be an isotropic variety over $k$, that is, a variety with a closed point of odd degree. Denote by $\mbp_{X}^{iso}$ the $E_{\infty}$-ring spectrum in $\SH(X)$ defined as $p_X^*(\mbp_k^{iso})$, where $p_X^*:\SH(k) \rightarrow \SH(X)$ is the pullback functor induced by the structure map $p_X:X \rightarrow \spec(k)$. We denote by $\mbp^{iso}_X - \Mod$ the stable $\infty$-category of modules over $\mbp^{iso}_X$, and by $\mbp^{iso}_X - \Mod_{cell}$ its full localizing subcategory generated by $\Sigma^{p,q}\mbp_X^{iso}$ for all $p,q \in \Z$. In our main result, we obtain a description of $\mbp^{iso}_X-\Mod_{cell}^{\omega}$, namely the full subcategory of compact objects of $\mbp^{iso}_X-\Mod_{cell}$.

\begin{thm}[Theorem \ref{main}]
	For every smooth isotropic variety $X$ over $k$, there is a non-degenerate $t$-structure on $\mbp^{iso}_X-\Mod_{cell}^{\omega}$ whose heart $\mbp^{iso}_X-\Mod_{cell}^{\omega,\heartsuit}$ is a neutral Tannakian category with fiber functor given by
	$$\bigoplus_n \gr_n^W:\mbp^{iso}_X-\Mod_{cell}^{\omega,\heartsuit} \rightarrow \Fi_2-\vs.$$
\end{thm}

The Tannaka group $\giso(X)$ associated with the fiber functor just introduced is an extension of the multiplicative group $\gm$ by a pro-unipotent group $\uiso(X)$ over $\Fi_2$ that we call isotropic motivic fundamental group of $X$. In fact, we get a split short exact sequence of pro-algebraic groups over $\Fi_2$
$$1 \rightarrow \uiso(X) \rightarrow \giso(X) \rightarrow \gm \rightarrow 1$$
that allows to interpret objects in the heart of the motivic $t$-structure on isotropic Tate motives over $X$ as graded representations of $\uiso(X)$ with values in $\Fi_2$-vector spaces. We point out that a similar result is also available for any isotropic finitely generated field extension $F/k$, that is, a function field of an isotropic smooth variety. This, in turn, produces an isotropic motivic Galois group ${\mathcal{U}}(F/k)$, and the heart of the motivic $t$-structure aforementioned consists of its graded representations.

The techniques we use to prove our main theorem are inspired by the work of Levine on motivic fundamental groups developed in \cite{L}, and later in \cite{L1}. In those papers, the object of study was the category of Tate motives over some base $X$ with coefficients in $\Q$. With the assumption of the Beilinson-Soul\'e vanishing conjecture, this category was equipped with a motivic $t$-structure whose heart is a Tannakian category over $\Q$. It would be tempting to try to apply these techniques directly to isotropic Tate motives in $\mathcal{DM}(k/k,\Fi_2)$. Unfortunately, the needed vanishing conditions are not met, and this is one of the main reasons why we turn our attention to cellular isotropic $\mbp$-modules. Indeed, this category, not only meets the standard vanishing conditions of a category of Tate type, but also satisfies the Beilinson-Soul\'e vanishing conjecture.

At this point, we can wonder whether $\mbp^{iso}_X-\Mod_{cell}^{\omega}$ and $\mbp_{F/k}-\Mod_{cell}^{\omega}$ are equivalent to the derived categories of their hearts. In the first situation, we know that it cannot be always the case, since the isotropic $\mbp$-cohomology of a variety spreads in general above the diagonal (i.e. the subring consisting of cohomology groups whose topological degree and motivic weight coincide). But for finitely generated field extensions, this does not happen. Indeed, we know that the isotropic $\mbp$-cohomology of a field extension $F/k$ is concentrated on the diagonal, and can be described as a specific quotient of Milnor K-theory that we call isotropic Milnor K-theory, and denote by $\km^M_*(F/k)$. In this situation it is clear from the work of Positselski \cite{P} that $\mbp_{F/k}-\Mod_{cell}^{\omega}$ is equivalent to the derived category of the heart of the motivic $t$-structure if and only if the isotropic Milnor K-ring of $F/k$ is a Koszul algebra. 

\begin{conj}[Conjecture \ref{kos}]
	For every finitely generated field extension $F/k$, the isotropic Milnor $K$-ring $\km^M_*(F/k)$ is a Koszul $\Fi_2$-algebra.
\end{conj}

We make this conjecture, motivated by the fact that, independently on the base flexible field, it holds for all finite extensions and purely transcendental extensions of degree one. Moreover, a similar Koszulity conjecture for the Milnor K-ring mod $p$ of a field was already proposed, and proved in some cases, by Positselski and Vishik in \cite{PV}.

Next, we perform some explicit computations in the cases of the punctured projective line and of split tori. More precisely, let $S$ be a non-empty finite subset of $k$-rational points of $\Pone$ of cardinality $d+1$, and denote by $T^c(V)$ the tensor coalgebra of a $d$-dimensional $\Fi_2$-vector space $V$. Also denote by $\Gamma(W)$ the divided power algebra of an $n$-dimensional $\Fi_2$-vector space $W$. Then, we obtain the following result.

\begin{thm}[Theorems \ref{pones} and \ref{uisogmn}]
	There are isomorphisms of pro-algebraic groups over $\Fi_2$
	$$\uiso(\Pone \setminus S) \cong \spec(T^c(V)), \: \: \: \uiso(\gmn) \cong \spec(\Gamma(W))$$
	that induce, respectively, $t$-exact equivalences of stable $\infty$-categories
	$$\mbp^{iso}_{\Pone \setminus S}-\Mod_{cell}^{\omega} \simeq {\mathcal{D}}^b({\mathrm{gr.rep}}_{\Fi_2}(\spec(T^c(V)))), \: \: \: 
	\mbp^{iso}_{\gmn}-\Mod_{cell}^{\omega} \simeq {\mathcal{D}}^b({\mathrm{gr.rep}}_{\Fi_2}(\spec(\Gamma(W)))).$$
\end{thm}

Finally, we point out that Spitzweck \cite{S} generalized Levine's techniques by removing the restrictions on the coefficients and on the Beilinson-Soul\'e vanishing conjecture, and by considering an affine derived group scheme and a suitable category of representations, which is identified then with the category of Tate motives. Of course, once one considers $\Q$-coefficients and assumes Beilinson-Soul\'e vanishing conjecture, then the motivic $t$-structure on Tate motives introduced by Levine is available, and the 0-truncation of Spitzweck's derived group coincides with Levine's motivic fundamental group. Guided by these considerations, we produce for any smooth isotropic variety $X$ (isotropic finitely generated field extension $F/k$) an affine derived group scheme $B^{\bullet,iso}_X$ ($B^{\bullet}_{F/k}$) over $\Fi_2$, and identify its category of representations with the respective category of relative isotropic Tate motives. 

\begin{thm}[Theorem \ref{der}]
	For every smooth isotropic variety $X$ over $k$, there is a symmetric monoidal equivalence
	$$\mbp_X^{iso}-\Mod_{cell}^{\omega} \simeq \perf(B^{iso,\bullet}_X).$$ 
\end{thm}

The 0-truncation of the derived group $B^{\bullet,iso}_X$ ($B^{\bullet}_{F/k}$) gives back the aforementioned isotropic motivic fundamental group (isotropic motivic Galois group). Then, we can see the Koszulity conjecture on isotropic Milnor K-theory in a different light: in fact, it reduces to the property of $B^{\bullet}_{F/k}$ of being already 0-truncated.\\

\textbf{Outline.} In Section 2, we fix notations that we repeatedly use in this paper. Section 3 recalls results from \cite{L} in a slightly different fashion that will be suitable to deal with isotropic Tate motives. We also summarize some facts about Koszul algebras and Koszul duality. In Section 4, we introduce the main objects of interest of this paper, namely the relative categories of isotropic Tate motives. Then, in Section 5, we introduce isotropic Milnor K-theory and relate it to isotropic $\mbp$-cohomology of a finitely generated field extension. Section 6 describes the picture for isotropic Tate motives over finite extensions, which allows to prove the Koszulity conjecture in a few cases. Section 7 proves the main theorem and introduces isotropic motivic fundamental groups, while in Section 8 we compute the examples of the punctured projective line and of split tori. Finally, in Section 9 we apply Spitzweck's derived approch to the categories of isotropic Tate motives by producing affine derived group schemes whose 0-truncations are the isotropic fundamental groups previously defined.\\

\textbf{Acknowledgements.} I would like to thank Timo Richarz and Alexander Vishik for useful comments that helped to improve the exposition and fix some mistakes. I am also grateful to Tom Bachmann, Rizacan Ciloglu, Konstantin Jakob and Can Yaylali for helpful conversations on the topic of this article. Finally, I thank the referee for the valuable comments, which have contributed to improving the paper.

\section{Notation}

\begin{tabular}{c|c}
	$\cat$ & symmetric monoidal stable $\infty$-category\\
	$\un$ & unit object in $\cat$\\
	$\Fi$ & field of coefficients\\
	$\Fi-\vs$ & category of finite dimensional $\Fi$-vector spaces\\
	${\mathrm{rep}}_{\Fi}(G)$ & category of finite-dimensional representations of an affine algebraic group $G$\\
	${\mathrm{gr.rep}}_{\Fi}(G)$ & category of finite-dimensional graded representations of an affine algebraic group $G$\\
	${\mathrm{Rep}}_{\Fi}(G)$ & category of representations of an affine algebraic group $G$\\
	${\mathcal{D}}^b({\mathcal{A}})$ & bounded derived category of an abelian category $\mathcal{A}$\\
	$k$ & flexible field with ${\mathrm{char}}(k) \neq 2$\\
	$\SH(k)$ & stable motivic homotopy category over $k$\\
	$\un_k^{iso}$ & isotropic sphere spectrum\\
	$\SH^{iso}(k)$ & isotropic stable motivic homotopy category over $k$\\
	$\mbp$ & motivic Brown-Peterson spectrum at the prime $2$\\
	$\mbp^{iso}_X$ & isotropic $\mbp$-spectrum over a smooth variety $X$\\
	$\mbp_{iso}^{**}(X)$ & isotropic $\mbp$-cohomology of a smooth variety $X$\\
	$\hz$ & motivic Eilenberg-MacLane spectrum with $\Z/2$-coefficients\\
	$\hz_{iso}^{**}(X)$ & isotropic motivic cohomology with $\Z/2$-coefficients of a smooth variety $X$\\
	$\un_{F/k}$ & isotropic sphere spectrum of a field extension $F/k$\\
	$\SH(F/k)$ & isotropic stable motivic homotopy category of a field extension $F/k$\\
	$\mbp_{F/k}$ & isotropic $\mbp$-spectrum over a field extension $F/k$\\
	$\km^M_*(F)$ & Milnor K-theory mod 2 of a field $F$\\
	$\km^M_*(F/k)$ & isotropic Milnor K-theory of a field extension $F/k$\\
	${\underline \pi}_n(E)_m$ & motivic homotopy sheaf of a motivic spectrum $E$\\
	$\Fi_2-\Vs_{**}$ & stable $\infty$-category of bigraded $\Fi_2$-vector spaces\\
	$\uiso(X)$ & isotropic motivic fundamental group of a smooth variety $X$\\
	${\mathcal{U}}(F/k)$ & isotropic motivic Galois group of a field extension $F/k$\\
	\end{tabular}\\

We denote by $[-,-]_{\cat}$ the hom-sets in the homotopy category of $\cat$, namely
$$[-,-]_{\cat} \coloneqq {\mathrm{Hom}}_{{\mathrm{Ho}}(\cat)}(-,-)=\pi_0{\mathrm{Map}}_{\cat}(-,-).$$

The bigraded suspension of a motivic spectrum $E$ is denoted by $\Sigma^{p,q}E$. We also use the notation $E(q)$ for the Tate twist $\Sigma^{0,q}E$.

For any motivic $E_{\infty}$-ring spectrum $E$, the stable $\infty$-category of $E$-modules is denoted by $E-\Mod$, while the stable $\infty$-category of cellular $E$-modules, that is, the full localizing subcategory of $E-\Mod$ generated by $\Sigma^{p,q}E$ for all $p, q \in \Z$, is denoted by $E-\Mod_{cell}$.

By abuse of notation, we sometimes call isotropic Tate motives over $X$ the objects in $\mbp_X^{iso}-\Mod_{cell}$ (see Section 4 for more details).

\section{Levine's categories of Tate type}

In this section, we start by recalling definitions and results from \cite{L} about categories of Tate type. Although originally stated for $\Q$-coefficients, all the results in \cite[Section 1]{L} work with coefficients in any field $\Fi$. Since we will need this slightly more general setting, we report here the main statements from \cite[Section 1]{L} by replacing $\Q$ with $\Fi$. All the proofs work the same way.

\begin{dfn}[3, Definition 1.1]\label{lctt}
\normalfont
    A symmetric monoidal $\Fi$-linear stable $\infty$-category $\cat$ is of Tate type if it is compactly generated by objects $\un(n)$ for $n \in \Z$, where $\un(0) \coloneqq \un$, $\un(1)$ is tensor-invertible, and $\un(n) \coloneqq \un(1)^{\otimes n}$, satisfying the following properties:

    1) $[\un(n),\Sigma^l\un(m)]_{\cat}\cong 0$ for all $l \in \Z$ and $n > m$;
    
    2) $[\un(n),\Sigma^l\un(n)]_{\cat}\cong 0$ for $l \neq 0$ and all $n \in \Z$;
    
    3) $[\un(n),\un(n)]_{\cat}\cong \Fi$ for all $n \in \Z$.
\end{dfn}

\begin{rem}
\normalfont
Let $\cat^{\omega}$ be the full subcategory of compact objects of $\cat$. Denote by $\cat^{\omega}_{[a,b]}$ the full stable subcategory of $\cat^{\omega}$ generated by $\un(n)$ for $a \leq -n \leq b$, where $a$ is allowed to be $-\infty$ and $b$ is allowed to be $+\infty$. Note that $\cat^{\omega}_a \coloneqq \cat^{\omega}_{[a,a]}$ is equivalent to the category of graded finite-dimensional $\Fi$-vector spaces. 
\end{rem} 

\begin{thm}[3, Lemma 1.2 and Definition 1.3]
    The pair $(\cat^{\omega}_{(-\infty,b]},\cat^{\omega}_{[b,+\infty)})$ defines a $t$-structure on $\cat^{\omega}$ with truncation functors $W^{\leq b}:\cat^{\omega} \rightarrow \cat^{\omega}_{(-\infty,b]}$ and $W^{\geq b}:\cat^{\omega} \rightarrow \cat^{\omega}_{[b,+\infty)}$.
\end{thm}

Let $\gr_a^WM$ be the object $W^{\geq a}W^{\leq a}M$ in $\cat^{\omega}_a$. Denote by $\cat^{\omega,\leq 0}$ the full subcategory of $\cat^{\omega}$ with objects $M$ such that $\gr_n^W M$ is cohomologically concentrated in non-negative degrees for all $n\in \Z$, and by $\cat^{\omega,\geq 0}$ the full subcategory of $\cat^{\omega}$ with objects $M$ such that $\gr_n^W M$ is cohomologically concentrated in non-positive degrees for all $n\in \Z$.

\begin{dfn}
	\normalfont
	We say that a category of Tate type $\cat$ satisfies the Beilinson-Soul\'e vanishing condition if $[\un(n),\Sigma^l\un(m)]_{\cat}\cong 0$ for all $n<m$ and all $l \leq 0$.
\end{dfn}

We report now the main theorem from \cite{L} we will use in this paper, that is, the existence of a motivic $t$-structure on categories of Tate type satisfying the Beilinson-Soul\'e vanishing condition.

\begin{thm}[3, Theorem 1.4]\label{tann}
  Let $\cat$ be a category of Tate type satisfying the Beilinson-Soul\'e vanishing condition. Then, the pair $(\cat^{\omega,\leq 0},\cat^{\omega,\geq 0})$ defines a non-degenerate $t$-structure on $\cat^{\omega}$ whose heart $\cat^{\omega,\heartsuit}$ is a neutral Tannakian category with fiber functor given by
  $$\bigoplus_n \gr_n^W:\cat^{\omega,\heartsuit} \rightarrow \Fi-\vs.$$
\end{thm}

\begin{rem}
	\normalfont
	Under the hypotheses of Theorem \ref{tann}, for every object $M$ in $\cat^{\omega,\heartsuit}$ there exists a pair of integers $a \leq b$ and a weight filtration
	$$\gr_a^W M  \cong W^{\leq a}M \hookrightarrow W^{\leq a+1}M \hookrightarrow \dots \hookrightarrow W^{\leq b-1}M \hookrightarrow W^{\leq b}M \cong M$$
	with successive quotients given by $\gr_c^W M$ for $a \leq c \leq b$.
\end{rem}

\begin{rem}\label{uts}
\normalfont
By \cite[Lemma C.2.4.3]{Lu}, the $t$-structure on $\cat^{\omega}$ from Theorem \ref{tann} extends to $\cat \simeq {\mathrm{Ind}}(\cat^{\omega})$. If we denote by $G$ the Tannaka group associated with the fiber functor $\bigoplus_n \gr_n^W$, then we obtain equivalences of abelian categories
$$\cat^{\omega,\heartsuit} \simeq {\mathrm{rep}}_{\Fi}(G), \: \: \: \cat^{\heartsuit} \simeq {\mathrm{Rep}}_{\Fi}(G).$$
\end{rem}

Before proceeding, we need a few well-known facts about Koszul algebras and coalgebras. For a very useful r\'esum\'e we refer to \cite[Section 3.7]{LV} and \cite[Sections 1 and 2]{PV}.

Let $V$ be an $\Fi$-vector space, $T(V)$ the associated tensor algebra and $R \subseteq V \otimes V$. Then, $A(V,R) \coloneqq T(V)/(R)$ is the quadratic algebra generated by $V$ with relations in $R$. Dually, the quadratic coalgebra associated with the quadratic datum $(V,R)$ is the subcoalgebra of the tensor coalgebra $T^c(V)$ defined as
$$C(V,R) \coloneqq \Fi \oplus V \oplus R \oplus (R \otimes V \cap V \otimes R) \oplus \dots \oplus (\bigcap_{i+j=n-2} V^{\otimes i} \otimes R \otimes V^{\otimes j}) \oplus \dots.$$

In this case, we say that $C = A^{\text{!`}}$ is the Koszul dual coalgebra of $A$ and $A=C^{\text{!`}}$ is the Koszul dual algebra of $C$. We also denote by $A^!$ the Koszul dual algebra of $A$, that is, the linear dual $A^{\text{!`}*}$ of the coalgebra $A^{\text{!`}}$.

\begin{rem} \label{um}
	\normalfont
A graded coalgebra $C=\bigoplus_{n \in \N}C^{(n)}$ is called strictly graded if $C^{(0)} \cong \Fi$ and $C^{(1)}$ coincides with the set of primitive elements of $C$. By \cite[Theorem 12.1.4]{Sw}, any strictly graded coalgebra $C$ admits at most one graded algebra structure that makes $C$ a graded bialgebra. Since any subcoalgebra of a tensor coalgebra is strictly graded, this holds in particular for quadratic coalgebras. 
\end{rem}

Denote by $H_{**}(A) \coloneqq {\mathrm{Tor}}_*^A(\Fi,\Fi)$ the homology coalgebra of $A$ and by $H^{**}(C) \coloneqq {\mathrm{Ext}}_C^{*}(\Fi,\Fi)$ the cohomology algebra of $C$.

One of the possible definitions of Koszul algebras and coalgebras is the following.

\begin{dfn}
	\normalfont
	A quadratic algebra $A$ is called Koszul if there is an isomorphism of coalgebras $A^{\text{!`}} \cong H_{**}(A)$. Similarly, a quadratic coalgebra $C$ is called Koszul if there is an isomorphism of algebras $C^{\text{!`}} \cong H^{**}(C)$.
\end{dfn}

\begin{rem}
	\normalfont
	Note that a quadratic algebra is Koszul if and only if its Koszul dual coalgebra is Koszul.
\end{rem}

We finish this section with a result establishing the sufficient conditions for a category of Tate type satisfying the Beilinson-Soul\'e vanishing condition to be equivalent to the derived category of its heart. The underlying philosophy is based on results from \cite{P}.

\begin{thm}\label{kosz}
  Let $\cat$ be a category of Tate type satisfying the Beilinson-Soul\'e vanishing condition. Assume further that $[\un(n),\Sigma^l\un(m)]_{\cat}\cong 0$ for all $l\neq m-n$. Then, there exists an equivalence of stable $\infty$-categories
  $${\mathcal{D}}^b(\cat^{\omega,\heartsuit}) \rightarrow \cat^{\omega}$$
if and only if the graded algebra $A_*$, defined by $A_n \coloneqq [\un,\Sigma^n\un(n)]_{\cat}$, is Koszul.
\end{thm}
\begin{proof}
By Remark \ref{uts} we know that $\cat^{\heartsuit}$ is the abelian category of representations of a Tannaka group, and so it has enough injectives. Moreover, by \cite[Section 1.2, Theorem]{P}, the algebra $A_*$ is Koszul if and only if
$$\ext^n_{\cat^{\heartsuit}}(X,Y) \cong [X, \Sigma^nY]_{\cat}$$
for all $X$ and $Y$ in $\cat^{\heartsuit}$ and for all $n \geq 0$. In particular, if $Y$ is injective we have that $[X, \Sigma^nY]_{\cat} \cong 0$ for all $n > 0$. Then, \cite[Proposition 2.12]{GWX} implies that there exists an equivalence of stable $\infty$-categories
$${\mathcal{D}}^b(\cat^{\heartsuit}) \rightarrow \cat^{b}$$
where $\cat^{b}$ is the bounded stable subcategory of $\cat$. By restricting the previous equivalence to the subcategories of compact objects
$${\mathcal{D}}^b(\cat^{\omega,\heartsuit}) \simeq {\mathcal{D}}^b(\cat^{\heartsuit})^{\omega} \simeq \cat^{\omega}$$
we conclude the proof.
\end{proof}

\section{Relative isotropic motivic categories}

Let $k=k_0(t_1,t_2,\dots)$ be a flexible field, that is, a purely transcendental extension of infinite degree over some other field, of characteristic different from $2$. Now, we recall from \cite{V} the definition of anisotropic variety mod 2.

\begin{dfn}
	\normalfont
A smooth variety over $k$ is called anisotropic mod 2 if all its closed points have even degree. It is called isotropic mod 2 if it is not anisotropic mod 2.
\end{dfn}

\begin{rem}
\normalfont
We point out that the definition of anisotropic variety is available also for odd primes $p$ over fields of characteristic different from $p$. Hence, all the constructions we are going to introduce here are available also for odd primes.

Unfortunately, the results we obtain are only available for the prime 2. The only obstruction boils down to \cite[Question 3.9]{V}. A positive answer to this question would guarantee similar results to the ones contained in this paper also for odd primes. 

Anyways, since we are ultimately considering only the prime 2, we will sometimes simply call anisotropic (isotropic) varieties the ones that are actually anisotropic (isotropic) mod 2.
\end{rem} 

\begin{dfn}
	\normalfont
Denote by $\un^{iso}_k$ the isotropic sphere spectrum, that is, 
$$\un^{iso}_k \coloneqq \cof(\Sigma^{\infty}_+ \ce(Q) \rightarrow \un)$$
in $\SH(k)$, where $Q$ is the coproduct of all $k$-anisotropic varieties mod 2 and $\ce(Q)$ is its \v{C}ech simplicial scheme. 
\end{dfn}

We know from \cite[Proposition 6.1]{T1} that $\un_k^{iso}$ is an $E_{\infty}$-ring spectrum. Denote by $\SH^{iso}(k)$ the isotropic stable motivic homotopy category, namely $\SH^{iso}(k) \coloneqq \un^{iso}_k-\Mod$. This comes equipped with a smashing localization functor $L_k^{iso}:\SH(k) \rightarrow \SH^{iso}(k)$, and we will simply denote by $E^{iso}_k$ the spectrum $L_k^{iso}E \coloneqq \un^{iso}_k \wedge E$, for any spectrum $E$ in $\SH(k)$. We refer to \cite{T} and \cite{T1} for more details.

Recall from \cite[Proposition 7.1]{T} that $\mbp^{iso}_k$ is an $E_{\infty}$-ring spectrum, and so we may consider the stable $\infty$-category of $\mbp^{iso}_k$-modules. We call the cohomology theory represented by $\mbp^{iso}_k$ isotropic $\mbp$-cohomology, namely
$$\mbp_{iso}^{**}(X) \coloneqq [\Sigma^{\infty}_+X,\Sigma^{**}\mbp^{iso}_k]_{\SH(k)}$$
for any smooth scheme $X$ over $k$.

\begin{rem}
\normalfont
The cellular subcategory of $\mbp^{iso}_k-\Mod$, i.e. the full localizing subcategory generated by Tate objects, is completely identified with the category of bigraded $\F$-vector spaces by \cite[Theorem 7.4]{T}. 
\end{rem}

In this paper we are interested in studying relative versions of $\mbp^{iso}_k-\Mod$. More precisely, we consider the following situation. For every smooth variety $X$ over $k$, the structure map $p_X: X \rightarrow \spec(k)$ induces a pair of adjoint functors
$$p_{X,\#}:\SH(X)\leftrightarrows \SH(k): p_X^*.$$

\begin{dfn}
	\normalfont
Let $\mbp^{iso}_X$ be the spectrum $p_X^*(\mbp^{iso}_k)$ in $\SH(X)$. This is an $E_{\infty}$-ring spectrum, and we denote by $\mbp^{iso}_X-\Mod$ the stable $\infty$-category of modules over it. We also denote by $\mbp^{iso}_X-\Mod_{cell}$ its full localizing subcategory generated by $\Sigma^{m,n}\mbp^{iso}_X$ for all $m,n \in \Z$.
\end{dfn}

The main motivation for investigating $\mbp^{iso}_X-\Mod_{cell}$ comes from the fact that we want to understand isotropic $\mbp$-cohomology, and this theory is represented in $\mbp^{iso}_X-\Mod_{cell}$. In fact, we have the following result.

\begin{prop}\label{rep}
For every smooth variety $X$ over $k$, the endomorphisms of the unit object $\un \coloneqq \mbp^{iso}_X$ in $\mbp^{iso}_X-\Mod$ are given by
$$[\un,\Sigma^{m,n}\un]_{\mbp^{iso}_X-\Mod} \cong \mbp_{iso}^{m,n}(X).$$
\end{prop}
\begin{proof}
The result follows from the following chain of isomorphisms
\begin{align*}
\mbp_{iso}^{m,n}(X) & \coloneqq [\Sigma^{\infty}_+X,\Sigma^{m,n}\mbp^{iso}_k]_{\SH(k)}  \\
& \cong [p_{X,\#}\un,\Sigma^{m,n}\mbp^{iso}_k]_{\SH(k)}\\
& \cong [\un,\Sigma^{m,n}p_X^*(\mbp^{iso}_k)]_{\SH(X)}\\
& \cong [\un,\Sigma^{m,n}\un]_{\mbp^{iso}_X-\Mod}.
\end{align*}
\end{proof}

\begin{rem}\label{needit}
	\normalfont
From the proof of \cite[Lemma 9.1]{T} we know that there is an equivalence of isotropic spectra
$$\hz^{iso}_k \simeq \bigvee_{\alpha \in A} \Sigma^{p_{\alpha},q_{\alpha}}\mbp^{iso}_k$$
for some set $A$. For every smooth variety $X$ over $k$, this induces a sequence of isomorphisms
\begin{align*}
	\hz_{iso}^{**}(X) & \coloneqq [\Sigma^{\infty}_+X,\Sigma^{**}\hz^{iso}_k]_{\SH(k)}  \\
	& \cong [\Sigma^{\infty}_+X,\bigvee_{\alpha \in A}\Sigma^{p_{\alpha},q_{\alpha}}\Sigma^{**}\mbp^{iso}_k]_{\SH(k)}\\
& \cong \bigoplus_{\alpha \in A} \Sigma^{p_{\alpha},q_{\alpha}}[\Sigma^{\infty}_+X,\Sigma^{**}\mbp^{iso}_k]_{\SH(k)}\\
	& \cong \bigoplus_{\alpha \in A} \Sigma^{p_{\alpha},q_{\alpha}}\mbp^{**}_{iso}(X),
\end{align*}
which for $X=\spec(k)$ gives
$$\hz_{iso}^{**}(k) \cong  \bigoplus_{\alpha \in A} \Sigma^{p_{\alpha},q_{\alpha}}\mbp^{**}_{iso}(k) \cong \bigoplus_{\alpha \in A} \Sigma^{p_{\alpha},q_{\alpha}} \Fi_2.$$

Therefore, we obtain an isomorphism
$$\hz_{iso}^{**}(X) \cong \hz_{iso}^{**}(k)\otimes_{\F} \mbp_{iso}^{**}(X)$$
where $\hz_{iso}^{**}(k) \cong \Lambda_{\Fi_2}(r_i:i \geq 0)$, with generators $r_i \in \hz_{iso}^{-2^{i+1}+1,-2^i+1}(k)$, by \cite[Theorem 3.7]{V}.

By restricting our attention to the pure part we get 
$$\mbp_{iso}^{2q,q}(X) \cong \hz_{iso}^{2q,q}(X) \cong {\mathrm{CH}}^q_{num}(X)/2$$
where the last isomorphism follows from \cite[Theorem 4.12]{V1}. This means that algebraic cycles of $X$ modulo numerical equivalence with $\Fi_2$-coefficients are represented in $\mbp^{iso}_X-\Mod_{cell}$.
\end{rem}

In this paper we also study relative isotropic categories for finitely generated field extensions $F/k$. Denote by $p_F: \spec(F) \rightarrow \spec(k)$ the structure map, and by $p_F^*:SH(k)\rightarrow \SH(F)$ the induced pullback functor.

\begin{dfn}
	\normalfont
	Let $\un_{F/k}$ be the $E_{\infty}$-ring spectrum $p_F^*(\un^{iso}_k)$ in $\SH(F)$. Denote by $\SH(F/k)$ the stable $\infty$-category of modules over $\un_{F/k}$. We also denote by $\mbp_{F/k}-\Mod$ the stable $\infty$-category of modules over $\mbp_{F/k}\coloneqq p_F^*(\mbp^{iso}_k)$, and by $\mbp_{F/k}-\Mod_{cell}$ its full localizing subcategory generated by $\Sigma^{m,n}\mbp_{F/k}$ for all $m,n \in \Z$.
\end{dfn}

\begin{rem}\label{funcfi}
	\normalfont
	Note that, if $F$ is the function field of a smooth variety $X$ over $k$, then
	$$\mbp_{iso}^{**}(F) \coloneqq [\un,\Sigma^{**}\mbp_{F/k}]_{\SH(F)} \cong \colim_U [\un,\Sigma^{**}\mbp^{iso}_{U}]_{\SH(U)} \cong \colim_U \mbp_{iso}^{**}(U)$$
	where the colimit is taken over all open $U \subset X$. 
	
	Similarly, there is an isomorphism $\hz_{iso}^{**}(F)\cong \colim_U \hz_{iso}^{**}(U)$. Therefore, it follows from Remark \ref{needit} that
	$$\hz_{iso}^{**}(F) \cong \hz_{iso}^{**}(k)\otimes_{\F} \mbp_{iso}^{**}(F).$$
\end{rem}

\section{Isotropic Milnor K-theory}

The aim of this section is to define isotropic Milnor K-theory, and to compare it to isotropic $\mbp$-cohomology. This will eventually provide the needed vanishing conditions for studying certain categories of isotropic Tate motives.

\begin{dfn}
\normalfont
 For every finitely generated field extension $F/k$, denote by $\km^M_*(F/k)$ the quotient of the mod 2 Milnor K-theory $\km^M_*(F)$ by the ideal $I_k$ defined as the union of annihilators of the image of non-zero pure symbols from $\km^M_*(k)$. We call $\km^M_*(F/k)$ the isotropic Milnor K-theory of the extension $F/k$.   
\end{dfn}

\begin{prop}\label{isoan}
Let $k$ be a flexible field, $X$ a smooth variety over $k$ and $F$ its function field. Then, the following are equivalent:
\begin{enumerate}
	\item $X$ is anisotropic mod 2,
	\item  there is a non-zero pure symbol of $\km^M_*(k)$ vanishing in $\km^M_*(F)$,
	\item the homomorphism $\km^M_*(k) \rightarrow \km^M_*(F)$ is not injective,
	\item $\km^M_*(F/k)\cong 0$.
\end{enumerate}
\end{prop}
\begin{proof}
	We start by proving that (1) implies (2). Since $X$ is anisotropic mod 2 and $k$ is flexible, by \cite[Statement 3.1]{V} and \cite[Theorem 1.1]{HI}, there is a non-zero pure symbol $\alpha$ in $\km^M_*(k)$ such that $X$ is contained in the Pfister quadric $Q_{\alpha}$ associated with $\alpha$. Here, we crucially rely on the flexibility of $k$, which ensures that passing to purely transcendental extensions of finite degree does not change the base field. It follows that $Q_{\alpha}$ has a rational point over $F$, which means that $\alpha$ vanishes in $\km^M_*(F)$.
	
	The implication (2) $\Rightarrow$ (3) and the equivalence (2) $\Leftrightarrow$ (4) are obvious, so we only need to show that (3) implies (1). Assume that $X$ is isotropic, that is, it has a closed point $\spec(E)$ of odd degree. Since the restriction map $\km^M_*(k) \rightarrow \km^M_*(E)$ is injective and factors through $\km^M_*(X)\coloneqq H^0_{Zar}(X,\underline{\km}^M_*)$, we deduce that the homomorphism $\km^M_*(k) \rightarrow \km^M_*(X)$ is also injective. On the other hand, the restriction to the generic point $\km^M_*(X) \rightarrow \km^M_*(F)$ is a monomorphism, which implies that the composition $\km^M_*(k) \rightarrow \km^M_*(F)$ is injective as well. This completes the proof.
	\end{proof}

\begin{rem}\label{evenext}
	\normalfont
	The previous result implies in particular that $\km^M_*(F/k)\cong 0$ for any finite extension $F/k$ of even degree.
\end{rem}

Proposition \ref{isoan} suggests that, for the sake of understanding isotropic motives, only certain field extensions should be considered. This motivates the following definition.

\begin{dfn}
	\normalfont
	A finitely generated field extension $F/k$ over a flexible field $k$ is called isotropic if $F$ is the function field of an isotropic smooth variety $X$ over $k$.
	
	In other words, $F/k$ is isotropic if and only if the homomorphism $\km^M_*(k) \rightarrow \km^M_*(F)$ is injective.
\end{dfn}

\begin{prop}\label{struc}
    For every finitely generated field extension $F/k$, $\km^M_*(F/k)$ is the quotient of the tensor algebra $T(\km^M_1(F/k))$ over the $\Fi_2$-vector space $\km^M_1(F/k)$ modulo the ideal generated by $a \otimes (1-a)$ for $a$ and $1-a$ in $\km^M_1(F/k)$.
\end{prop}
\begin{proof}
Denote by $I_{k,1}$ the set of elements in degree 1 of the ideal $I_k \subset \km^M_*(F)$. It follows from \cite[Theorem 3.3]{OVV} that $I_k$ is the ideal generated by $I_{k,1}$. Hence, we get a commutative diagram
	$$
\xymatrix{
	0\ar@{->}[r] \ar@{->}[d]&(I_{k,1})  \ar@{->}[r] \ar@{->>}[d]_f& T(\km^M_1(F)) \ar@{->}[r] \ar@{->>}[d]^g& T(\km^M_1(F/k)) \ar@{->}[r] \ar@{->>}[d]^h& 0 \ar@{->}[d]\\
	0\ar@{->}[r] &I_k  \ar@{->}[r] & \km^M_*(F) \ar@{->}[r]& \km^M_*(F/k)\ar@{->}[r] &0
}
$$
where the rows are short exact sequences.

Since $f$ is surjective, we deduce by the snake lemma that the homomorphism $\ker(g) \rightarrow \ker(h)$ is surjective as well. Then, the result follows from the fact that $\ker(g)$ is the ideal generated by $a \otimes (1-a)$ for $a$ and $1-a$ in $\km^M_1(F)$.
\end{proof}

\begin{rem}
\normalfont
  Proposition \ref{struc} shows in particular that the isotropic Milnor K-ring of any extension $F/k$ is a quadratic $\Fi_2$-algebra. This follows in general from the fact that, if we take the quotient of a quadratic algebra by an ideal generated in degree 1, we still get a quadratic algebra.
\end{rem}

In \cite{P}, it is conjectured that the Milnor K-ring of a field with $\Fi_p$-coefficients is a Koszul $\Fi_p$-algebra. The conjecture is proved to be true for $\Q$ and finite fields in \cite{PV}. We make the same conjecture for isotropic Milnor K-theory.

\begin{conj}\label{kos}
 For every finitely generated field extension $F/k$, the isotropic Milnor $K$-ring $\km^M_*(F/k)$ is a Koszul $\Fi_2$-algebra.
\end{conj}

In the next section we prove Conjecture \ref{kos}, independently on the base flexible field $k$, for finite extensions $F$ and purely transcendental extensions $k(x)$ of degree 1.

Before proceeding with showing the deep relation between isotropic Milnor K-theory and isotropic $\mbp$-cohomology, we need to recall from \cite{M} the definition of the homotopy $t$-structure on $\SH(k)$. This is explicitly given by:
$$\SH(k)_{\geq 0} \coloneqq \{E \in SH(k) : \underline{\pi}_n(E)_m \cong 0 \: {\mathrm{for}} \:  n<0,m \in \Z\},$$
$$\SH(k)_{\leq 0} \coloneqq \{E \in SH(k) : \underline{\pi}_n(E)_m \cong 0 \: {\mathrm{for}}\: n>0,m \in \Z\},$$
where the motivic homotopy sheaf $\underline{\pi}_n(E)_m$ is the Nisnevich sheafification of the presheaf
$$\pi_n(E)_m(U) \coloneqq [\Sigma^{m+n,m}\Sigma^{\infty}_+U,E]_{\SH(k)}.$$

The heart of the homotopy $t$-structure $\SH(k)^{\heartsuit}$ consists of those spectra $E$ whose possibly non-vanishing motivic homotopy sheaves are only of the type $\underline{\pi}_0(E)_m$. Besides, $\SH(k)^{\heartsuit}$ is equivalent to the abelian category of homotopy modules described by Morel in \cite[Section 5.2]{M}.

We are now ready to show the main result of this section.

\begin{prop}\label{imk}
$\mbp^{iso}_k$ belongs to the heart of the homotopy $t$-structure. Moreover, the canonical map of ring spectra $\mbp^{iso}_k \rightarrow \hz^{iso}_k$ induces an isomorphism
$$\mbp_{iso}^{**}(F) \cong \km^M_*(F/k).$$
\end{prop}
\begin{proof}
   For every finitely generated field extension $F/k$, by \cite[Corollary 3.3]{V}, we get an isomorphism
 $$\hz_{iso}^{m,n}(F) \coloneqq [\un,\Sigma^{m,n}p^*_F\hz^{iso}_k]_{\SH(F)}  \cong \colim_{\alpha}[\un,\Sigma^{m,n}(\widetilde{\mathcal{X}}_{Q_{\alpha}} \wedge \hz)]_{\SH(F)}$$
   where $\alpha$ runs over all non-zero pure symbols from $\km^M_*(k)$, $Q_{\alpha}$ is the associated Pfister quadric, and $\widetilde{\mathcal{X}}_{Q_{\alpha}}$ is the cofiber of the map $\Sigma^{\infty}_+\ce(Q_{\alpha}) \rightarrow \un$ in $\SH(F)$. 
   
   Since $k$ is a flexible field, the previous colimit is filtered, again by \cite[Corollary 3.3]{V}, and by reproducing the same arguments of the proof of \cite[Theorem 3.7]{V}, we obtain an isomorphism
    $$\hz_{iso}^{**}(F) \cong \hz_{iso}^{**}(k)\otimes_{\F} \km^M_*(F/k).$$
    
    On the other hand, we know by Remark \ref{funcfi} that
$$\hz_{iso}^{**}(F) \cong \hz_{iso}^{**}(k)\otimes_{\F} \mbp_{iso}^{**}(F).$$
    
    We conclude that the map of ring spectra $\mbp^{iso}_k \rightarrow \hz^{iso}_k$ induces an isomorphism $\mbp_{iso}^{p,p}(F) \cong\hz_{iso}^{p,p}(F) \cong \km^M_p(F/k)$ for all $p \in \Z$. Moreover, we notice that
    $$\underline{\pi}_{p-q}(\mbp^{iso}_k)_q(F) \cong \mbp_{iso}^{p,q}(F) \cong 0$$ 
    for $p \neq q$, and so $\mbp^{iso}_k$ belongs to the heart of the homotopy $t$-structure.
\end{proof}

\begin{rem}\label{prev}
\normalfont
   More precisely, Proposition \ref{imk} tells us that there is an equivalence
   $$\mbp_k^{iso} \cong \underline{\pi}_0(\hz^{iso}_k)$$
   from which it follows that $\mbp^{iso}_k$ is a Rost cycle module by \cite[Th\'eor\`eme 3.7]{D}. Moreover, for any smooth variety $X$ over $k$, we have an isomorphism
   $$\mbp_{iso}^{p,q}(X) \cong H^{p-q}_{Zar}(X,\underline{\pi}_0(\hz^{iso}_k)_{-q})\cong H^{p-q}_{Zar}(X,\underline{\km}^M_q(-/k)).$$
\end{rem}

\begin{prop}\label{van}
Let $X$ be a smooth variety over $k$ of dimension $d$. Then,
$\mbp_{iso}^{p,q}(X)\cong 0$ for $p > 2q$, $p < q$ and $p > q+d$.
\end{prop}
\begin{proof}
We use the Gersten-Quillen spectral sequence for $\mbp^{iso}_k$, which explicitly looks as follows
$$E_1^{i,j} = \bigoplus_{x \in X^{(i)}}  \mbp^{j-i,q-i}_{iso}(\kappa(x)) \Longrightarrow \mbp^{j+i,q}_{iso}(X).$$

By Proposition \ref{imk}, the only non-trivial contribution of the $E_1$-page to $\mbp^{p,q}_{iso}(X)$ is $\mbp^{j-i,q-i}_{iso}(\kappa(x))$ for $j+i=p$ and $j-i=q-i$. This is simply $\km^M_{2q-p}(\kappa(x)/k)$ that vanishes for $p>2q$. The vanishing of $\mbp^{p,q}_{iso}(X)$ for $p<q$ and $p>q+d$ follows immediately from Remark \ref{prev}. This finishes the proof.
\end{proof}

\section{Isotropic Tate motives over finite extensions}

In this section, we focus on finite extensions. In particular, we identify the isotropic stable motivic homotopy category of a finite extension of odd degree with the isotropic stable motivic homotopy category of the trivial extension. This also allows to fully understand isotropic Tate motives over a finite extension.

We start with the following lemma about base change of anisotropic varieties along finite extensions of odd degree.

\begin{lem}\label{points}
Let $F/k$ be a finite separable extension of odd degree, let $X$ be a smooth variety over $F$ and let $Y$ be a smooth variety over $k$. Then, the following hold: 
\begin{enumerate}
\item $X$ is $k$-anisotropic mod 2 if and
only if $X$ is $F$-anisotropic mod 2,

\item if $Y$ is $k$-anisotropic mod 2, then $Y_F \coloneqq Y \times_{\spec(k)} \spec(F)$ is $F$-anisotropic mod 2.
\end{enumerate}
\end{lem}
\begin{proof}
Let $x$ be closed a point of $X$. Then, we have a sequence of extensions $k \rightarrow F \rightarrow \kappa(x)$. It follows that $\kappa(x)/F$ is a finite extension of odd degree if and only if $\kappa(x)/k$ is a finite extension of odd degree. This concludes the proof of (1).

For proving (2), Suppose that $Y_F$ has a closed point $x$ of odd degree. Then, we have a finite extension $F \rightarrow \kappa(x)$ of odd degree. It follows that also $\kappa(x)/k$ is finite of odd degree. Denote by $y$ the image of $x$ in $Y$. Then, $k \rightarrow \kappa(y) \rightarrow \kappa(x)$ is finite of odd degree from which it follows that $\kappa(y)/k$ is finite of odd degree, that is, $Y$ has a closed point of odd degree. This finishes the proof of (2).
\end{proof}

Before proceeding, we need the following well-known result (valid over any field $k$). We include a proof for completeness.

\begin{lem}\label{tech}
Let $X$ and $Y$ be smooth schemes over $k$. Then, the following are equivalent:
\begin{enumerate}
\item there is a simplicial weak equivalence $\ce(X)\simeq \ce(Y)$,
\item $X(R) \neq \emptyset$ if and only if $Y(R) \neq \emptyset$ for any Henselian local ring $R$ over $k$.
\item $X(L) \neq \emptyset$ if and only if $Y(L) \neq \emptyset$ for any field extension $L/k$.
\end{enumerate}
\end{lem}
\begin{proof}
	We start by proving that (1) implies (2). First, note that $\ce(X)(R)$ is the empty simplicial set if $X(R)= \emptyset$, and it is contractible otherwise, for any smooth scheme $X$ and any Henselian local ring $R$. Since it follows from (1) that $\ce(X)(R)\simeq \ce(Y)(R)$ is a weak equivalence for all Henselian local rings $R$, this immediately implies (2).
	
	To prove (2) $\Rightarrow$ (1), assume first that there is a map $X \rightarrow Y$. This induces a map $\ce(X) \rightarrow \ce(Y)$, and to show that this map is a simplicial weak equivalence we have to check that $\ce(X)(R) \rightarrow \ce(Y)(R)$ is a weak equivalence for all Henselian local rings $R$ over $k$. This immediately follows from (2). 
	
	If there is neither a map $X \rightarrow Y$ nor a map $Y \rightarrow X$, we can still consider the span $X \leftarrow X\times Y \rightarrow Y$. In this case (2) implies that $X(R) \neq \emptyset$ if and only if $(X\times Y)(R) \neq \emptyset$, from which it follows that $\ce(X\times Y) \rightarrow \ce(X)$ is a simplicial weak equivalence, by what we have shown before. Similarly, one proves that $\ce(X\times Y) \rightarrow \ce(Y)$ is a simplicial weak equivalence, which implies (1).
	
	To prove that (2) and (3) are equivalent, it is enough to show that $X(R) \neq \emptyset$ if and only if $X(\kappa) \neq \emptyset$, where $R$ is any Henselian local ring and $\kappa$ is its residue field. For this, we can assume that $X$ is affine, i.e. $X\cong \spec(A)$ for some smooth $k$-algebra $A$. Then, we have a surjection
	$$X(R) = {\mathrm{Hom}}_k(A,R) \twoheadrightarrow {\mathrm{Hom}}_k(A,\kappa)=X(\kappa)$$
	since $R$ is Henselian, which is exactly what we need to finish the proof.
	\end{proof}

Lemma \ref{points} and Lemma \ref{tech} can be directly applied to obtain an equivalence between the isotropic sphere spectrum of an odd degree extension $F/k$ and the isotropic sphere spectrum of the trivial extension $F/F$.

\begin{prop}\label{isoun}
 Let $F/k$ be a finite separable extension of odd degree. Then, $\un_{F/k} \cong \un^{iso}_F$ in $\SH(F)$.
\end{prop}
\begin{proof}
Denote by $Q$ the coproduct of $Y_F$ for all $k$-anisotropic varieties $Y$, and by $P$ the coproduct of all $F$-anisotropic varieties $X$. Since $p^*_F$ preserves both limits and colimits, being at the same time a right and a left adjoint, we deduce that there is an equivalence
$$\un_{F/k} \cong {\mathrm{cofib}}(\Sigma^{\infty}_+\ce(Q) \rightarrow \un) $$
in $\SH(F)$, while, by definition,
$$\un^{iso}_F \coloneqq {\mathrm{cofib}}(\Sigma^{\infty}_+\ce(P) \rightarrow \un).$$

Hence, it is enough to show that $\Sigma^{\infty}_+\ce(Q) \cong \Sigma^{\infty}_+\ce(P)$ in $\SH(F)$. For this, by Lemma \ref{tech}, it suffices to
prove that $Q(L) \neq \emptyset$ if and only if $P(L) \neq \emptyset$ for any field extension $L/F$. By Lemma \ref{points}, we have a map $Q \rightarrow P$, and so $Q(L) \neq \emptyset$ implies $P(L) \neq \emptyset$. On the other hand, suppose we have a morphism $\spec(L) \rightarrow X$ where $X$ is $F$-anisotropic. Since we also have a map $\spec(L) \rightarrow \spec(F)$, we finally obtain a map $\spec(L) \rightarrow X_F$, where $X$ is $k$-anisotropic by Lemma \ref{points}. Therefore, $P(L) \neq \emptyset$ implies $Q(L) \neq \emptyset$, and we are done. 
\end{proof}

\begin{cor}\label{relsh}
	For every finite separable extension of odd degree $F/k$, there is an equivalence of stable $\infty$-categories
	$$\SH(F/k) \simeq \SH_{}^{iso}(F).$$
\end{cor}
\begin{proof}
	This is an immediate consequence of Proposition \ref{isoun}.
	\end{proof}

\begin{rem}\label{dah}
\normalfont
By considering cellular isotropic $\mbp$-modules, we deduce from Corollary \ref{relsh} and \cite[Theorem 7.4]{T} an equivalence of stable $\infty$-categories
$$\mbp_{F/k}-\Mod_{cell} \simeq \Fi_2-{\mathrm{Vec}}_{**}$$
for every separable finite extension $F/k$ of odd degree. 
\end{rem}

What we have proved so far implies the following description of the isotropic Milnor K-theory of a finite extension of odd degree.

\begin{prop}\label{vanfe}
If $F/k$ is a finite extension of odd degree, then $\km^M_*(F/k)\cong \km^M_*(F/F)\cong \km^M_*(k/k)$, that is $\F$ in degree 0, and 0 otherwise. In particular, it is a Koszul algebra.
\end{prop}
\begin{proof}
If $F/k$ is a purely inseparable extension, then $\km^M_*(F)\cong \km^M_*(k)$, from which the statement immediately follows.

On the other hand, if $F/k$ is a separable extension, then by Propositions \ref{isoun} and \ref{imk} we obtain the following chain of isomorphisms
\begin{align*}
\km^M_*(F/k)&\cong [\Sigma^{\infty}_+\spec(F),\Sigma^{**}\mbp_k^{iso}]_{\SH(k)} \\
&\cong [p_{F,\#}\un,\Sigma^{**}\mbp^{iso}_k]_{\SH(k)} \\
&\cong [\un,\Sigma^{**}p_F^*(\mbp^{iso}_k)]_{\SH(F)} \\
&\cong [\un_,\Sigma^{**}\mbp^{iso}_F]_{\SH(F)} \cong \km^M_*(F/F)    
\end{align*}
which concludes the proof.
\end{proof}

Proposition \ref{vanfe} immediately implies the following result about vanishing of Milnor K-groups of purely transcendental extensions of degree 1.

\begin{cor}\label{pte}
If $k(x)/k$ is a purely transcendental extension of degree 1 over a flexible field $k$, then $\km^M_p(k(x)/k)\cong 0$ for $p\geq 2$. In particular, $\km^M_*(k(x)/k)$ is a Koszul algebra.
\end{cor}
\begin{proof}
By the homotopy property of Rost cycle modules \cite[Proposition 2.2]{R}, we have a short exact sequence of isotropic Milnor K-groups
$$0 \rightarrow \km^M_*(k/k) \rightarrow \km^M_*(k(x)/k) \rightarrow \bigoplus_{F/k \: {\mathrm{finite}}} \km^M_{*-1}(F/k)\rightarrow 0.$$

Then, the vanishing of the middle term follows from the vanishing of the two side terms for $* \geq 2$ given by Proposition \ref{vanfe} and Remark \ref{evenext}.
\end{proof}

\begin{rem}
	\normalfont
Corollary \ref{pte} provides also some information about the structure of Milnor K-theory mod 2 of flexible fields. In particular, it says that every 2-symbol in $\km^M_*(k(x))$ belongs to the annihilator of some non-zero pure symbol in $\km^M_*(k)$. We expect this to be true more generally, that is, every $n+1$-symbol in $\km^M_*(k(x_1,\dots,x_n))$ is expected to belong to the annihilator of some non-zero pure symbol in $\km^M_*(k)$.
\end{rem}

\begin{cor}
	Let $X$ be a smooth variety over $k$ of dimension $d$. Then,
	$\mbp_{iso}^{2d+q,d+q}(X)\cong 0$ for $q>0$.
\end{cor}
\begin{proof}
It follows from Proposition \ref{vanfe} and Remark \ref{evenext} by analysing the Gersten-Quillen spectral sequence for $\mbp^{iso}_k$. In fact, the non-trivial contribution of the $E_1$-page to $\mbp_{iso}^{2d+q,d+q}(X)$ is given by $\mbp^{j-i,d+q-i}_{iso}(\kappa(x))$ for $j+i=2d+q$ and $j-i=d+q-i$. This implies $i=d$, and so $$\mbp^{j-i,d+q-i}_{iso}(\kappa(x)) \cong \mbp^{q,q}_{iso}(\kappa(x)) \cong \km^M_q(\kappa(x)/k)$$ that vanishes for $q>0$ since $\kappa(x)/k$ is a finite extension. 
\end{proof}

\section{Isotropic motivic fundamental groups}

We are now ready to prove the main theorem of this paper that applies Levine's formalism of categories of Tate type to cellular isotropic $\mbp$-modules. This endows $\mbp^{iso}_X-\Mod_{cell}^{\omega}$ with a motivic $t$-structure whose heart is a Tannakian category. Studying the associated Tannaka group is the main purpose of these last sections.

\begin{thm}\label{main}
	For every smooth isotropic variety $X$ over $k$, there is a non-degenerate $t$-structure on $\mbp^{iso}_X-\Mod_{cell}^{\omega}$ whose heart $\mbp^{iso}_X-\Mod_{cell}^{\omega,\heartsuit}$ is a neutral Tannakian category with fiber functor given by
	$$\bigoplus_n \gr_n^W:\mbp^{iso}_X-\Mod_{cell}^{\omega,\heartsuit} \rightarrow \Fi_2-\vs.$$
\end{thm}
\begin{proof}
	By definition, $\mbp^{iso}_X-\Mod_{cell}$ is compactly generated by Tate twists of the unit $\un(n)\coloneqq \Sigma^{0,n}\un$, so we only need to check the 3 properties from Definition \ref{lctt}, and the Beilinson-Soul\'e vanishing condition.
	
	By Propositions \ref{rep} and \ref{van} we know that
	$$[\un(n),\Sigma^l\un(m)]_{\mbp^{iso}_X-\Mod}\cong \mbp_{iso}^{l,m-n}(X) \cong 0$$
	for all $l \in \Z$ and $m <n$,
	$$[\un(n),\Sigma^l\un(n)]_{\mbp^{iso}_X-\Mod}\cong \mbp_{iso}^{l,0}(X) \cong 0$$
	for $l \neq 0$ and all $n \in \Z$,
	$$[\un(n),\un(n)]_{\mbp^{iso}_X-\Mod}\cong \mbp_{iso}^{0,0}(X) \cong \km^M_{0}(\kappa(X)/k) \cong \Fi_2$$
	for all $n \in \Z$, and 
	$$[\un(n),\Sigma^l\un(m)]_{\mbp^{iso}_X-\Mod}\cong \mbp_{iso}^{l,m-n}(X) \cong 0$$
	for all $m >n$ and $l \leq 0$. Hence, $\mbp^{iso}_X-\Mod_{cell}$ is an $\Fi_2$-linear stable $\infty$-category of Tate type satisfying the Beilinson-Soul\'e vanishing condition. Then, we conclude by Theorem \ref{tann}.
\end{proof}

We have also the following analogous statement for finitely generated isotropic field extensions.

\begin{thm}
	For every finitely generated isotropic field extension $F/k$, there is a non-degenerate $t$-structure on $\mbp_{F/k}-\Mod_{cell}^{\omega}$ whose heart $\mbp_{F/k}-\Mod_{cell}^{\omega,\heartsuit}$ is a neutral Tannakian category with fiber functor given by
	$$\bigoplus_n \gr_n^W:\mbp_{F/k}-\Mod_{cell}^{\omega,\heartsuit} \rightarrow \Fi_2-\vs.$$
\end{thm}
\begin{proof}
	The proof is the same of Theorem \ref{main}. 
\end{proof}

We denote by $\giso(X)$ the Tannaka group of $\mbp^{iso}_X-\Mod_{cell}^{\omega,\heartsuit}$, and by ${\mathcal{G}}(F/k)$ the Tannaka group of $\mbp_{F/k}-\Mod_{cell}^{\omega,\heartsuit}$. In other words, there are equivalences of abelian categories
$$\mbp^{iso}_X-\Mod_{cell}^{\omega,\heartsuit} \simeq {\mathrm{rep}}_{\Fi_2}(\giso(X)),$$
and similarly
$$\mbp_{F/k}-\Mod_{cell}^{\omega,\heartsuit} \simeq {\mathrm{rep}}_{\Fi_2}({\mathcal{G}}(F/k)).$$

\begin{prop}
	For every smooth isotropic variety $X$ over $k$, there is a split short exact sequence of pro-algebraic groups over $\Fi_2$
	$$1 \rightarrow \uiso(X) \rightarrow \giso(X) \rightarrow \gm \rightarrow 1$$
	where $\uiso(X)$ is pro-unipotent. 
\end{prop}
\begin{proof}
	First, note that, since $\mbp^{iso}_k-\Mod_{cell}^{\omega} \simeq \Fi_2-{\mathrm{vec}}_{**}$, the Tannaka group $\giso(k)$ is simply $\gm$. In this case, the fiber functor
	$$\bigoplus_n \gr_n^W:\mbp^{iso}_k-\Mod_{cell}^{\omega,\heartsuit} \rightarrow \Fi_2-\vs$$
	associated with the $t$-structure and weight filtration guaranteed by Theorem \ref{main} is the functor from graded $\Fi_2$-vector spaces to $\Fi_2$-vector spaces that forgets the grading provided by the Tate twist.
	
	On the other hand, the functor $p_X^*$ is $t$-exact with respect to the $t$-structure in Theorem \ref{main} and also preserves the weight filtration. Therefore, $p_X^*$ restricts on the hearts, and the fiber functor for $X$ restricts to the fiber functor for $k$. In other words, we get a composite 
	$${\mathrm{rep}}_{\Fi_2}(\gm) \xrightarrow{p_X^*} {\mathrm{rep}}_{\Fi_2}(\giso(X)) \xrightarrow{\bigoplus_n \gr_n^W} {\mathrm{rep}}_{\Fi_2}(\gm),$$
    that is an equivalence. This induces a map $\giso(X) \rightarrow \gm$ of pro-algebraic groups over $\Fi_2$ endowed with a splitting. Let $\uiso(X)$ be the kernel of $\giso(X) \rightarrow \gm$. Then, we get a split short exact sequence
	$$1 \rightarrow \uiso(X) \rightarrow \giso(X) \rightarrow \gm \rightarrow 1.$$
	
	Under the equivalences
	$$\mbp^{iso}_X-\Mod_{cell}^{\omega,\heartsuit} \simeq {\mathrm{rep}}_{\Fi_2}(\giso(X)) \simeq {\mathrm{gr.rep}}_{\Fi_2}(\uiso(X)),$$ 
	an object in $\mbp^{iso}_X-\Mod_{cell}^{\omega,\heartsuit}$ corresponds to a graded finite-dimensional $\Fi_2$-vector space $V=\bigoplus_{i=m}^nV_i$, where the grading encodes the motivic weight, endowed with an action of $\uiso(X)$. Since this action is trivial on each graded component $V_i$, it follows that $\uiso(X)$ is pro-unipotent, which completes the proof.
\end{proof}

\begin{rem}
	\normalfont
	By the same arguments, the previous result holds also for ${\mathcal{G}}(F/k)$, where $F/k$ is a finitely generated isotropic field extension.
\end{rem}

\begin{dfn}
	\normalfont
	We call $\uiso(X)$ the isotropic motivic fundamental group of $X$, and ${\mathcal{U}}(F/k)$ the isotropic motivic Galois group of the extension $F/k$.
\end{dfn}

\begin{rem}
	\normalfont
	Note that, for every finite extension $F/k$ of odd degree, Remark \ref{dah} implies that ${\mathcal{G}}(F/k) \cong \gm$, and so ${\mathcal{U}}(F/k)$ is the trivial group.
\end{rem}

The next result shows how the Koszulity conjecture for isotropic Milnor K-theory is related to a precise description of the category of cellular isotropic $\mbp$-modules for field extensions.

\begin{prop}
	Conjecture \ref{kos} is true if and only if there is an equivalence of stable $\infty$-categories
	$$\mbp_{F/k}-\Mod_{cell}^{\omega} \simeq {\mathcal{D}}^b({\mathrm{rep}}_{\Fi_2}({\mathcal{G}}(F/k))).$$
\end{prop}
\begin{proof}
	We already know that $\mbp_{F/k}-\Mod_{cell}$ is a stable $\infty$-category of Tate type satisfying the Beilinson-Soul\'e vanishing condition. Moreover, Proposition \ref{imk} implies
	$$[\un(n),\Sigma^l\un(m)]_{\mbp_{F/k}-\Mod}\cong \mbp_{iso}^{l,m-n}(F) \cong 0$$
	for $l \neq m-n$, and
	$$[\un,\Sigma^n\un(n)]_{\mbp_{F/k}-\Mod}\cong \mbp_{iso}^{n,n}(F) \cong \km^M_n(F/k).$$
	
	Then, the result follows immediately from Theorem \ref{kosz}.
\end{proof}

\section{The isotropic motivic fundamental groups of $\Pone \setminus S$ and $\gmn$}

Let $S$ be a non-empty finite subset of $k$-rational points of $\Pone$ of cardinality $d+1$. In this section we compute the isotropic motivic fundamental group of $\Pone \setminus S$, and deduce from this computation the structure of the category of isotropic Tate motives over $\Pone \setminus S$.

We start by computing the isotropic $\mbp$-cohomology of $\Pone \setminus S$.

\begin{lem}\label{pi-s}
There is an isomorphism of $\Fi_2$-algebras
   $$\mbp_{iso}^{**}(\Pone \setminus S) \cong \Fi_2[x_1,\dots,x_d]/(x_ix_j: 1 \leq i,j \leq d)$$
   with generators $x_i \in \mbp_{iso}^{1,1}(\Pone \setminus S)$ for $1 \leq i \leq d$.
\end{lem}
\begin{proof}
	In $\mbp_k^{iso}-\Mod$ there is a Gysin triangle
	$$\Sigma^{\infty}_+(\Pone \setminus S) \wedge \mbp^{iso}_k \rightarrow \Sigma^{\infty}_+\Pone \wedge \mbp^{iso}_k \rightarrow \bigvee_{i=1}^{d+1} \Sigma^{2,1}\mbp^{iso}_k \rightarrow \Sigma^{1,0}(\Sigma^{\infty}_+(\Pone \setminus S) \wedge \mbp^{iso}_k)$$
	that induces, after applying the cohomological functor $[-,\Sigma^{**}\mbp^{iso}_k]_{\mbp_k^{iso}-\Mod}$, a long exact sequence of bigraded $\Fi_2$-vector spaces
	$$\dots \rightarrow \mbp_{iso}^{p-1,q}(\Pone \setminus S) \rightarrow  \bigoplus_{i=1}^{d+1}(\Sigma^{2,1}\Fi_2)^{p,q}\rightarrow (\Fi_2 \oplus \Sigma^{2,1}\Fi_2)^{p,q} \rightarrow  \mbp_{iso}^{p,q}(\Pone \setminus S) \rightarrow \dots.$$
	
	It follows that $\mbp_{iso}^{0,0}(\Pone \setminus S) \cong \Fi_2$, $\mbp_{iso}^{1,1}(\Pone \setminus S) \cong \bigoplus_{i=1}^d \Fi_2\cdot x_i$ and $\mbp_{iso}^{p,q}(\Pone \setminus S)\cong 0$ in all other degrees. The ring structure follows authomatically. This concludes the proof.
\end{proof}

Denote by $V$ the vector space $\bigoplus_{i=1}^d \Fi_2\cdot x_i$ and by $T^c(V)$ the associated tensor coalgebra (with the deconcatenation coproduct). The latter is a Hopf algebra with the shuffle algebra structure. Note that the shuffle product is commutative.

\begin{thm}\label{pones}
    There is an isomorphism of pro-algebraic groups over $\Fi_2$
    $$\uiso(\Pone \setminus S) \cong \spec(T^c(V))$$
    that induces a $t$-exact equivalence of stable $\infty$-categories
    $$\mbp^{iso}_{\Pone \setminus S}-\Mod_{cell}^{\omega} \simeq {\mathcal{D}}^b({\mathrm{rep}}_{\Fi_2}(\giso(\Pone \setminus S)))$$
where $\giso(\Pone \setminus S) \cong \uiso(\Pone \setminus S) \rtimes \gm$.
\end{thm}
\begin{proof}
By Lemma \ref{pi-s}, we know that
$$[\un(n),\Sigma^l\un(m)]_{\mbp^{iso}_{\Pone \setminus S}-\Mod}\cong \mbp_{iso}^{l,m-n}(\Pone \setminus S) \cong 0$$
 for all $l\neq m-n$, and
 $$[\un,\Sigma^n\un(n)]_{\mbp^{iso}_{\Pone \setminus S}-\Mod}\cong \mbp_{iso}^{n,n}(\Pone \setminus S)$$
 is the $n$-graded component of the Koszul algebra $\Fi_2[x_1,\dots,x_d]/(x_ix_j: 1 \leq i,j \leq d)$. Hence, by Theorem \ref{kosz}, there is a $t$-exact equivalence of stable $\infty$-categories 
 $$\mbp^{iso}_{\Pone \setminus S}-\Mod_{cell}^{\omega} \simeq {\mathcal{D}}^b({\mathrm{rep}}_{\Fi_2}(\giso(\Pone \setminus S))).$$
 
 It only remains to identify the Tannaka group $\giso(\Pone \setminus S)$. For this, note that there is an equivalence of abelian categories
 $${\mathrm{rep}}_{\Fi_2}(\uiso(\Pone \setminus S)) \simeq {\mathrm{comod}}(C)$$
 where $C$ is the Hopf algebra such that $\uiso(\Pone \setminus S) \cong \spec(C)$, and the cohomology of $C$ is given by
 $$H^{**}(C) \cong \mbp^{**}_{iso}(\Pone \setminus S) \cong \Fi_2[x_1,\dots,x_d]/(x_ix_j: 1 \leq i,j \leq d).$$
 
Therefore, by Koszul duality, $C$ is the Koszul dual coalgebra of $\Fi_2[x_1,\dots,x_d]/(x_ix_j: 0 \leq i,j \leq d)$, i.e. the tensor coalgebra $T^{c}(V)$. By Remark \ref{um}, the isomorphism of coalgebras $C \cong T^c(V)$ upgrades to an isomorphism of Hopf algebras that is what we aimed to show. 
\end{proof}

Now we consider the analogous case of split tori. Again, we start by computing the isotropic $\mbp$-cohomology of $\gmn$.

\begin{lem}\label{gmn}
	There is an isomorphism of $\Fi_2$-algebras
	$$\mbp_{iso}^{**}(\gmn) \cong \Lambda_{\Fi_2}(x_1,\dots,x_n)$$
	with generators $x_i \in \mbp_{iso}^{1,1}(\gmn)$ for $1 \leq i \leq n$.
\end{lem}
\begin{proof}
	In $\mbp^{iso}_k-\Mod$, we have an identification
	$$\Sigma^{\infty}_+\gm \wedge \mbp_k^{iso}\cong \mbp_k^{iso} \vee \Sigma^{1,1}\mbp_k^{iso},$$
	and so, by taking the smash $n$-power, we obtain
	$$\Sigma^{\infty}_+\gmn \wedge \mbp_k^{iso}\cong \mbp_k^{iso} \vee (\bigvee_{1 \leq i \leq n} \Sigma^{1,1}\mbp_k^{iso}) \vee (\bigvee_{1 \leq i<j \leq n} \Sigma^{2,2}\mbp_k^{iso}) \vee \dots \vee \Sigma^{n,n}\mbp_k^{iso}.$$
	
	By applying the cohomological functor $[-,\Sigma^{**}\mbp^{iso}_k]_{\mbp_k^{iso}-\Mod}$ to the previous equivalence, we get an isomorphism of $\Fi_2$-algebras
	$$\mbp_{iso}^{**}(\gmn) \cong \Lambda_{\Fi_2}(x_1,\dots,x_n)$$
	that concludes the proof.
\end{proof}

Let $W$ be the vector space $\bigoplus_{i=1}^n \Fi_2\cdot x_i$, and denote by $\Gamma(W)$ the divided power Hopf algebra of $W$, namely the graded algebra of invariants $\bigoplus_{m \in \N}(W^{\otimes m})^{\Sigma_m}$.

\begin{thm}\label{uisogmn}
	There is an isomorphism of pro-algebraic groups over $\Fi_2$
	$$\uiso(\gmn) \cong \spec(\Gamma(W))$$
	that induces a $t$-exact equivalence of stable $\infty$-categories
	$$\mbp^{iso}_{\gmn}-\Mod_{cell}^{\omega} \simeq {\mathcal{D}}^b({\mathrm{rep}}_{\Fi_2}(\giso(\gmn)))$$
	where $\giso(\gmn) \cong \uiso(\gmn) \rtimes \gm$.
\end{thm}
\begin{proof}
	 Here, we follow the lines of the proof of Theorem \ref{pones}. In this case, there is an equivalence of abelian categories
	$${\mathrm{rep}}_{\Fi_2}(\uiso(\gmn)) \simeq {\mathrm{comod}}(C)$$
	where $C$ is the Hopf algebra satisfying $\uiso(\gmn) \cong \spec(C)$ and
	$$H^{**}(C) \cong \mbp^{**}_{iso}(\gmn) \cong \Lambda_{\Fi_2}(x_1,\dots,x_n)$$
	is a Koszul algebra. It follows by Koszul duality that $C$ is the Koszul dual coalgebra of $\Lambda_{\Fi_2}(x_1,\dots,x_n)$, and so $C^*$ is the Koszul dual algebra of $\Lambda_{\Fi_2}(x_1,\dots,x_n)$. This implies that $C^*\cong\Fi_2[x_1^*,\dots,x_n^*]\cong S(W^*)$ as $\Fi_2$-algebras, where $S(W^*)$ is the symmetric algebra of the dual of $W$. Note that $S(W^*)$ is a Hopf algebra, and its linear dual is the divided power Hopf algebra $\Gamma(W)$. Hence, we get an isomorphism of coalgebras $C \cong \Gamma(W)$ that upgrades to an isomorphism of Hopf algebras by Remark \ref{um}. This concludes the proof.
	\end{proof}

\begin{rem}
	\normalfont                                                     
	Theorem \ref{uisogmn} tells us that isotropic Tate motives over $\gmn$, modulo the grading provided by the Tate twist, are essentially given by modules over the polynomial algebra $\Fi_2[x_1^*,\dots,x_n^*]$.
\end{rem}

We conclude by explicitly mentioning the case of isotropic Tate motives over $\gm$.

\begin{cor}\label{gmga}
    There is an isomorphism of pro-algebraic groups over $\Fi_2$
    $$\uiso(\gm) \cong \spec(\Fi_2[y_1,y_2,\dots,y_{2^k},\dots]/(y_1^2,y_2^2,\dots,y_{2^k}^2,\dots))$$
    that induces a $t$-exact equivalence of stable $\infty$-categories
    $$\mbp^{iso}_{\gm}-\Mod_{cell}^{\omega} \simeq {\mathcal{D}}^b({\mathrm{rep}}_{\Fi_2}(\uiso(\gm) \rtimes \gm)).$$
\end{cor}
\begin{proof}
    It follows as a special case of Theorem \ref{uisogmn} after noticing that the divided power algebra of a one-dimensional $\Fi_2$-vector space is isomorphic to $\Fi_2[y_1,y_2,\dots,y_{2^k},\dots]/(y_1^2,y_2^2,\dots,y_{2^k}^2,\dots)$.
\end{proof}

\section{Derived isotropic motivic fundamental groups}

In this last section, by using techniques developed by Spitzweck in \cite{S}, we provide a derived approach to isotropic fundamental groups. This will ensure an identification of the category of isotropic Tate motives with the category of representations of an affine derived group scheme. The 0-truncation of this derived group gives back the isotropic fundamental group guaranteed by Theorem \ref{main}.

First, fix a smooth variety $X$ over $k$. Since the functor $\mbp^{iso}_k-\Mod_{cell} \hookrightarrow \mbp^{iso}_k-\Mod$ preserves colimits it has a right adjoint $c:\mbp^{iso}_k-\Mod \rightarrow \mbp^{iso}_k-\Mod_{cell}$ that we call cellularization. Then, the adjunction
$$p_X^*: \mbp^{iso}_k-\Mod \leftrightarrows \mbp^{iso}_X-\Mod:p_{X,*}$$ 
induces, in turn, an adjunction on the subcategories of Tate objects
$$p_X^*: \mbp^{iso}_k-\Mod_{cell} \leftrightarrows \mbp^{iso}_X-\Mod_{cell}:\overline{p}_{X,*}$$
where $\overline{p}_{X,*}=c \circ p_{X,*}$.

Since $p_X^*$ is symmetric monoidal, its right adjoint $\overline{p}_{X,*}$ is automatically lax symmetric monoidal. Therefore, the previous adjunction restricts to commutative algebra objects
$$p_X^*: {\mathrm{CAlg}}(\mbp^{iso}_k-\Mod_{cell}) \leftrightarrows {\mathrm{CAlg}}(\mbp^{iso}_X-\Mod_{cell}):\overline{p}_{X,*}$$
by \cite[Proposition 5.22]{MNN}.

\begin{dfn}
	\normalfont
Denote by $A^{iso}_X$ the $E_{\infty}$-algebra in $\mbp^{iso}_k-\Mod_{cell} \simeq \Fi_2-\Vs_{**}$ defined as $\overline{p}_{X,*}p_X^*\un \cong \overline{p}_{X,*}\un$. In other words, $A^{iso}_X$ is the cellularization of the cohomological isotropic motive of $X$.
\end{dfn}

\begin{prop}\label{bbl}
	For every smooth isotropic variety $X$ over $k$, there is an equivalence of symmetric monoidal stable $\infty$-categories 
	$$\mbp^{iso}_X-\Mod_{cell} \simeq \Mod_{\Fi_2-\Vs_{**}}(A^{iso}_X).$$
\end{prop}
\begin{proof}
The functor $p_X^*: \mbp^{iso}_k-\Mod_{cell} \rightarrow \mbp^{iso}_X-\Mod_{cell}$ is symmetric monoidal and preserves compact objects. Hence, its right adjoint $\overline{p}_{X,*}$ commutes with colimits, since the categories in question are compactly generated.

Now, take any object $M$ in $\mbp^{iso}_X-\Mod_{cell}$ and $N$ in $\mbp^{iso}_k-\Mod_{cell} \simeq \Fi_2-\Vs_{**}$. Then, $N$ is a filtered colimit of finite dimensional $\Fi_2$-vector spaces $N_{\alpha} \cong \bigoplus_{i=1}^{n_{\alpha}} \Sigma^{p_{\alpha,i},q_{\alpha,i}}\Fi_2$ from which it follows that
$$p^*_X(N_{\alpha}) \cong \bigoplus_{i=1}^{n_{\alpha}} \Sigma^{p_{\alpha,i},q_{\alpha,i}}\un \cong \bigoplus_{i=1}^{n_{\alpha}} \Sigma^{p_{\alpha,i},q_{\alpha,i}}\mbp^{iso}_X,$$
and so we get a chain of equivalences
\begin{align*}
	\overline{p}_{X,*}(M \otimes p_X^*(N))  &\cong \overline{p}_{X,*}(M \otimes p_X^*({\mathrm{colim}}_{\alpha} N_{\alpha})) \\
	&\cong {\mathrm{colim}}_{\alpha}\overline{p}_{X,*}(M \otimes p^*_X(N_{\alpha})) \\
	&\cong {\mathrm{colim}}_{\alpha}\bigoplus_{i=1}^{n_{\alpha}} \Sigma^{p_{\alpha,i},q_{\alpha,i}}\overline{p}_{X,*}(M)\\
	&\cong \overline{p}_{X,*}(M) \otimes N.
	\end{align*}

This means that the adjunction $(p_X^*,\overline{p}_{X,*})$ satisfies the projection formula. Moreover, the right adjoint $\overline{p}_{X,*}$ is conservative. In fact, $\overline{p}_{X,*}(M) \cong 0$ implies
$$[\Sigma^{m,n}\un,M]_{\mbp^{iso}_X-\Mod_{cell}} \cong [\Sigma^{m,n}p_X^*\un,M]_{\mbp^{iso}_X-\Mod_{cell}} \cong [\Sigma^{m,n}\un,\overline{p}_{X,*}(M)]_{\mbp^{iso}_k-\Mod_{cell}} \cong 0$$
for all $m,n \in \Z$, and so $M \cong 0$.
We conclude by applying the symmetric monoidal version of Barr-Beck-Lurie \cite[Proposition 5.29]{MNN}.
	\end{proof}

\begin{dfn}
	\normalfont
We say that an $E_{\infty}$-algebra $A$ in $\Fi_2-\Vs_{**}$ is coconnective if the unit $\Fi_2 \rightarrow A_{0,*}$ is an equivalence, and the complex $A_{n,*}$ is trivial for all $n >0$.

We say that $A$ is of Tate-type if the unit $\Fi_2 \rightarrow A_{*,0}$ is an equivalence, and the complex $A_{*,n}$ is trivial for all $n > 0$. It is moreover called bounded if for every $n < 0$ the complex $A_{*,n}$ is cohomologically bounded from below.
\end{dfn}

\begin{lem}\label{btt}
For every smooth isotropic variety $X$ over $k$, the $E_{\infty}$-algebra $A^{iso}_X$ is coconnective of bounded Tate-type.
\end{lem}
\begin{proof}
	By definition of $A^{iso}_X$ and Proposition \ref{rep}, we know that
	$$(A^{iso}_X)_{p,q} \coloneqq [\Sigma^{p,q}\un,\overline{p}_{X,*}\un]_{\mbp^{iso}_k-\Mod_{cell}} \cong [\Sigma^{p,q}\un,\un]_{\mbp^{iso}_X-\Mod_{cell}} \cong \mbp^{-p,-q}_{iso}(X).$$
	
	Then, the vanishing of isotropic $\mbp$-cohomology groups provided by Proposition \ref{van} and the isomorphism $\mbp_{iso}^{0,0}(X) \cong \Fi_2$ imply that $A^{iso}_X$ is coconnective of bounded Tate-type. 
	\end{proof}

\begin{dfn}
	\normalfont
	Denote by $B^{iso,\bullet}_X$ the \v{C}ech conerve of the augmentation $A^{iso}_X \rightarrow \Fi_2$, that is, the cosimplicial $A^{iso}_X$-algebra whose $n$-th component is given by
	$$\Fi_2^{\otimes_{A^{iso}_X}(n+1)} \coloneqq \Fi_2 \otimes_{A^{iso}_X} \dots \otimes_{A^{iso}_X} \Fi_2$$
	(tensor product of $n+1$ factors).
\end{dfn}

\begin{rem}
	\normalfont
	Recall from \cite[Section 5]{S} that $B^{iso,\bullet}_X$ is an affine derived group scheme. We denote by $\perf(B^{iso,\bullet}_X)$ the category of perfect representations of $B^{iso,\bullet}_X$ defined in \cite[Section 3]{S}. This is essentially the limit in $\widehat{\mathrm{Cat}}_{\infty}$ given by
	$$\perf(B^{iso,\bullet}_X):=\lim_n \perf(B^{iso,n}_X),$$
	and consists of those cosimplicial modules $M^{\bullet}$ such that $M^n$ belongs to $\perf(B^{iso,n}_X)$ and, for any map $[n] \rightarrow [m]$ in the simplex category, the corresponding morphism $M^n \otimes_{B^{iso,n}_X}B^{iso,m}_X \rightarrow M^m$ is an isomorphism.
	
	Besides, note that, as a cosimplicial algebra, $B^{iso,\bullet}_X$ is coconnective by \cite[Corollary 4.1.11]{Lu1}.
\end{rem}

\begin{thm}\label{der}
For every smooth isotropic variety $X$ over $k$, there is a symmetric monoidal equivalence
$$\mbp_X^{iso}-\Mod_{cell}^{\omega} \simeq \perf(B^{iso,\bullet}_X).$$ 
\end{thm}
\begin{proof}
	By restricting the equivalence in Proposition \ref{bbl} to the subcategories of compact objects we get an equivalence
	$$\mbp_X^{iso}-\Mod_{cell}^{\omega} \simeq \perf(A^{iso}_X).$$
	
	On the other hand, it follows from Lemma \ref{btt} and \cite[Theorem 7.4]{S} that
	$$\perf(A^{iso}_X) \simeq \perf(B^{iso,\bullet}_X),$$
    which concludes the proof.
	\end{proof}

\begin{rem}
	\normalfont
	By \cite[Lemma C.2.4.3]{Lu}, the $t$-structure on $\mbp_X^{iso}-\Mod_{cell}^{\omega}$ provided by Theorem \ref{main} induces a $t$-structure on $\mbp_X^{iso}-\Mod_{cell} \simeq {\mathrm{Ind}(\mbp_X^{iso}-\Mod_{cell}^{\omega})}$.  Via the equivalence in Proposition \ref{bbl}, the latter corresponds to the canonical $t$-structure (in cohomological notation) on $\Mod_{\Fi_2-\Vs_{**}}(A^{iso}_X)$, whose non-positive part $\Mod_{\Fi_2-\Vs_{**}}(A^{iso}_X)^{\leq 0}$ is generated by $A^{iso}_X(n)$ for $n \in \Z$ under colimits and extensions, and whose non-negative part $\Mod_{\Fi_2-\Vs_{**}}(A^{iso}_X)^{\geq 0}$ consists of modules $M$ such that $\pi_iM \cong 0$ for all $i >0$. Hence, we get equivalences of tensor abelian categories
	$$\mbp_X^{iso}-\Mod_{cell}^{\omega,\heartsuit} \simeq \perf(A_X^{iso})^{\heartsuit} \simeq \perf(B_X^{iso,\bullet})^{\heartsuit}.$$
	
	Since $\perf(B_X^{iso,\bullet})^{\heartsuit}$ consists of cosimplicial modules $M^{\bullet}$ such that $M^0$ is an object of $\perf(\Fi_2)$ concentrated in degree 0, namely a finite-dimensional $\Fi_2$-vector space, we actually have equivalences
	$$\perf(B_X^{iso,\bullet})^{\heartsuit}\simeq {\mathrm{gr.comod}}(H^0(B_X^{iso,1})) \simeq {\mathrm{gr.rep}}_{\Fi_2}({\mathcal{U}}_{A^{iso}_{X}}),$$
	where $H^0(B_X^{iso,1})$ is the 0-truncation of the coconnective $E_{\infty}$-Hopf algebra $B_X^{iso,1}=\Fi_2 \otimes_{A^{iso}_X}  \Fi_2$, ${\mathcal{U}}_{A^{iso}_{X}} \coloneqq \spec(H^0(B_X^{iso,1}))$ is the corresponding affine group scheme over $\Fi_2$, and the extra-grading comes from the motivic weight. At the end, we get an identification of pro-algebraic groups
	$$\uiso(X) \cong {\mathcal{U}}_{A^{iso}_{X}}$$
	underlying the connection between the classical and the derived approach.
\end{rem}

\begin{rem}
	\normalfont
	As a final comment, we point out that the results in this section also work verbatim for any isotropic finitely generated field extension $F/k$. Thus, we get a coconnective $E_{\infty}$-algebra $A_{F/k}$ of bounded Tate-type, with the associated affine derived group scheme $B^{\bullet}_{F/k}$, such that
	$$\mbp_{F/k}-\Mod_{cell} \simeq \Mod_{\Fi_2-\Vs_{**}}(A_{F/k})$$
	and
	$$\mbp_{F/k}-\Mod_{cell}^{\omega} \simeq \perf(B^{\bullet}_{F/k}).$$ 
	
	In this context, Conjecture \ref{kos} is equivalent to saying that the affine derived group scheme $B^{\bullet}_{F/k}$ is 0-truncated.
	\end{rem}

\footnotesize{

}

\noindent {\scshape Dipartimento di Matematica e Applicazioni, Universit\`a degli Studi di Milano-Bicocca}\\
fabio.tanania@gmail.com

\end{document}